\documentclass[12pt,twoside]{amsart}

\usepackage{amsmath,amsthm,amscd,amssymb,mathrsfs,graphicx,amsfonts,mathrsfs}
\usepackage{amsfonts}
\usepackage{amssymb,enumerate}
\usepackage{amsthm}
\usepackage[all]{xy}
\usepackage{hyperref}
\allowdisplaybreaks[4] \footskip=12pt
\renewcommand{\uppercasenonmath}[1]{}

\numberwithin{equation}{section} \theoremstyle{plain}
\newtheorem*{thm*}{Main Theorem}
\newtheorem{thm}{Theorem}[section]
\newtheorem{cor}[thm]{Corollary}
\newtheorem*{cor*}{Corollary}
\newtheorem{lem}[thm]{Lemma}
\newtheorem*{lem*}{Lemma}

\newtheorem*{fact*}{Fact}

\newtheorem*{nota*}{Notation}

\newtheorem*{prop*}{Proposition}
\newtheorem{rem}[thm]{Remark}
\newtheorem*{rem*}{Remark}

\newtheorem*{observation*}{Observation}

\newtheorem*{exa*}{Example}

\newtheorem*{df*}{Definition}

\newtheorem*{con*}{Construction}

\begin{document}
\begin{center}
{\large  \bf  Cosupport of artinian modules}

\vspace{0.5cm} Xiaoyan Yang\\
Department of Mathematics, Northwest Normal University, Lanzhou 730070,
China
E-mail: yangxy@nwnu.edu.cn
\end{center}

\bigskip
\centerline { \bf  Abstract}
\leftskip10truemm \rightskip10truemm \noindent Let $(R,\mathfrak{m})$ be a commutative local noetherian ring. For an artinian $R$-module $M$, we show the equality \begin{center}$\mathrm{cosupp}_RM=\mathrm{Cosupp}_RM$,\end{center} using the semi-discrete linearly compactness of $R$-module $\mathrm{Hom}_R(R_\mathfrak{p},M)$ where $\mathfrak{p}$ is a prime ideal of $R$.\\
\vbox to 0.3cm{}

\leftskip0truemm \rightskip0truemm
\bigskip
\section* { \bf Introduction}
Given a commutative
ring $R$ and a multiplicatively closed subset $S\subseteq R$, Melkersson and Schenzel \cite{MS} called the $S^{-1}R$-module $\mathrm{Hom}_R(S^{-1}R,M)$ the co-localization of an $R$-module $M$ which is dual to the ordinary localization
$S^{-1}M=S^{-1}R\otimes_RM$. They proved in this paper that their construction
works well when $M$ is artinian. For example, the
co-localization of an artinian $R$-module has a secondary representation and the functor $\mathrm{Hom}_R(S^{-1}R,-)$ is an exact functor on the category of artinian $R$-modules.

Related to the support of a module defined in terms of the localization, Melkersson and Schenzel \cite{MS}
also defined the cosupport of a module in terms of the co-localization.
The cosupport of an $R$-module $M$ as  \begin{center}$\mathrm{Cos}_RM=\{\mathfrak{p}\in\mathrm{Spec}R\hspace{0.03cm}|\hspace{0.03cm}\mathrm{Hom}_R(R_\mathfrak{p},M)\neq0\}$.\end{center} For an artinian $R$-module $M$ we have that $M\neq0$ if and only if $\mathrm{Cos}_RM\neq\emptyset$. But this does not
hold in general. Therefore, Yassemi \cite{Y} introduced the cocyclic modules and defined a sensible cosupport for non-artinian modules, which coincided with $\mathrm{Cos}_RM$
when $M$ is artinian. An $R$-module $L$ is cocyclic if $L$ is a submodule of $E(R/\mathfrak{m})$ for some $\mathfrak{m}\in\mathrm{Max}R$, where $E(R/\mathfrak{m})$ is the injective envelope
 of $R/\mathfrak{m}$. The cosupport of $M$ is defined as the set
of prime ideals $\mathfrak{p}$ such that there is a cocyclic homomorphic image $L$
of $M$ such that $\mathfrak{p}\subseteq\mathrm{Ann}_RL$, the annihilator of $L$, and denoted this set by $\mathrm{Cosupp}_RM$.

It is well-known that the localization $S^{-1}N$ of a finitely generated $R$-module $N$ is a finitely generated $S^{-1}R$-module. Therefore, one can obtain an equality
\begin{center}$\mathrm{supp}_RN=\mathrm{Supp}_RN=\mathrm{V}(\mathrm{Ann}_RN)$\end{center} by Nakaymam's lemma. However, the co-localization $\mathrm{Hom}_R(S^{-1}R,M)$ of an artinian $R$-module $M$ is almost never an artinian $S^{-1}R$-module, and it may even have infinite Goldie dimension (see \cite[Example 3.8]{CNh}). The
 purpose of this paper is to prove an analogue of the above equality for cosupport. That is,

\vspace{2mm} \noindent{\bf Theorem.}\label{Th1.4} {\it{Let $(R,\mathfrak{m})$ be a commutative local noetherian ring. For an artinian $R$-module $M$, there exists an equality \begin{center}$\mathrm{cosupp}_RM=\mathrm{Cosupp}_RM=\mathrm{V}(\mathrm{Ann}_RM)$.\end{center}}}
\vspace{2mm}

\bigskip
\section{\bf Preliminaries}
Unless stated to the contrary we assume throughout this paper that $R$ is a commutative noetherian ring with non-zero identity.

This section is devoted to recalling some notions and basic facts
for use throughout this paper. For terminology we shall follow \cite{BS} and \cite{Y}.

We write $\mathrm{Spec}R$ for the set of
prime ideals of $R$ and $\mathrm{Max}R$ for the set of
maximal ideals of $R$, For an ideal $\mathfrak{a}$ in $R$, we set
\begin{center}$\mathrm{V}(\mathfrak{a})=\{\mathfrak{p}\in\textrm{Spec}R\hspace{0.03cm}|\hspace{0.03cm}\mathfrak{a}\subseteq\mathfrak{p}\}$.
\end{center}

Let $M$ be an $R$-module. The set $\mathrm{Ass}_RM$ of associated prime of $M$ is
the set of prime ideals $\mathfrak{p}$ of $R$
such that there exists a cyclic submodule $N$ of $M$ such that $\mathfrak{p}=\mathrm{Ann}_RN$. The set of prime ideals $\mathfrak{p}$ such that there exists a cyclic submodule $N$
of $M$ with $\mathfrak{p}\supseteq\mathrm{Ann}_RN$ is well-known to be the support of $M$, denoted by $\mathrm{Supp}_RM$, that is the set
 \begin{center}$\{\mathfrak{p}\in\mathrm{Spec}R\hspace{0.03cm}|\hspace{0.03cm}
M_\mathfrak{p}\neq0\}$.\end{center}
A non-zero $R$-module $M$ is said to be secondary if for any
$x\in R$, the multiplication by $x$ on $M$ is either surjective or nilpotent. The radical
of the annihilator of $M$ is then a prime ideal $\mathfrak{p}$ and we say that $M$ is $\mathfrak{p}$-secondary.
A secondary representation of an $R$-module $M$ is an expression of
$M$ as a sum $M=M_1+\cdots+M_n$ of $\mathfrak{p}_i$-secondary submodules. This
representation is said to be minimal if the prime ideals $\mathfrak{p}_i$ are all distinct and none of the summands $M_i$ are redundant. Note that any secondary representation of $M$
can be refined to a minimal one. The set $\{\mathfrak{p}_1,\cdots,\mathfrak{p}_n\}$ is then independent of
the choice of minimal representation of $M$. This set is denoted by $\mathrm{Att}_RM$ and
called the set of attached prime ideals of $M$.

The concept of
linear compactness was first introduced by Lefschetz \cite{SL} for vector spaces of
arbitrary dimension and extended for modules by Zelinsky \cite{DZ} and Leptin \cite{HL}, which plays an important role for duality in algebra.

Let $M$ be a topological $R$-module. $M$ is said to be linearly topologized if $M$ has a
base of neighborhoods of the zero element $\mathcal{M}$ consisting of submodules. $M$ is called
Hausdorff if the intersection of all the neighborhoods of the zero element is 0. A
Hausdorff linearly topologized $R$-module $M$ is said to be linearly compact if $\mathcal{F}$ is a
family of closed cosets (i.e., cosets of closed submodules) in $M$ which has the finite
intersection property, then the cosets in $\mathcal{F}$ have a non-empty intersection (see \cite{Ma}).

\begin{lem}\label{lem:0.1}{\it{Let $M$ be a linearly compact $R$-module. One has
\begin{center}$\mathrm{Cos}_RM=\mathrm{Cosupp}_RM$.\end{center}In particular, $N\neq0$ if and only if $\mathrm{Cos}_RN\neq\emptyset$ for any linearly compact $R$-module $N$.}}
\end{lem}
\begin{proof}  Let
 By \cite[Corollary 2.16]{Y}, $\mathrm{Cos}_RM\subseteq\mathrm{Cosupp}_RM$. On the other hand, let $\mathfrak{p}\in\mathrm{Cosupp}_RM$. There exists an exact sequence
\begin{center}$0\rightarrow N\rightarrow M\rightarrow A\rightarrow0$,\end{center}
where $A$ is artinian with $\mathrm{Ann}_RA\subseteq\mathfrak{p}$. Thus, $\mathfrak{p}\in\mathrm{Cosupp}_RA$. Since $A$ is a linearly compact $R$-module with the discrete topology, $N$
is an open submodule of $M$, and hence $N$ is closed in $M$. It means that
$N$ is a linearly compact submodule of $M$. Hence the sequence $0\rightarrow \mathrm{Hom}_R(R_\mathfrak{p},N)\rightarrow \mathrm{Hom}_R(R_\mathfrak{p},M)\rightarrow \mathrm{Hom}_R(R_\mathfrak{p},A)\rightarrow0$ is exact by \cite[Corollary 2.5]{CNh}.
Since $\mathrm{Hom}_R(R_\mathfrak{p},A)\neq0$ by \cite[7.3]{MS}, we have $\mathrm{Hom}_R(R_\mathfrak{p},M)\neq0$, as desired.
\end{proof}

Let $M$ be an $R$-module. The `small' support of $M$ is the set \begin{center}$\mathrm{supp}_RM=\{\mathfrak{p}\in\mathrm{Spec}R\hspace{0.03cm}|\hspace{0.03cm}
\mathrm{Tor}^R_i(R/\mathfrak{p},M)_\mathfrak{p}\neq0\ \textrm{for\ some}\ i\}$.\end{center} By the remark after \cite[Theorem 3.8]{Y}, one has \begin{center}$\mathfrak{p}\in\mathrm{Cosupp}_RM \Longleftrightarrow \mathrm{Hom}_{R}(\bigoplus_\mathfrak{m}\mathrm{Hom}_R(M,E(R/\mathfrak{m})),
E(R/\mathfrak{p}))=:{^\mathfrak{p}}M\neq0$.\end{center}
For an $R$-module $M$, we define the `small' cosupport of $M$ is the set
\begin{center}$\mathrm{cosupp}_RM=\{\mathfrak{p}\in\mathrm{Spec}R\hspace{0.03cm}|\hspace{0.03cm}
{^\mathfrak{p}}\mathrm{Ext}^i_R(R/\mathfrak{p},M)\neq0\ \textrm{for\ some}\ i\}$,\end{center}which is the set of prime ideals $\mathfrak{p}$ such that $\mathrm{Ext}^i_R(R/\mathfrak{p},{^\mathfrak{p}}M)\neq0$ for some $i$.

A Hausdorff linearly topologized $R$-module $M$ is called semi-discrete if every
submodule of $M$ is closed. The class of
semi-discrete linearly compact $R$-modules contains all artinian $R$-modules and all finitely generated modules over a complete ring.

\begin{lem}\label{lem:0.2}{\it{Let $M$ be a semi-discrete linearly compact $R$-module. One has
\begin{center}$\mathrm{cosupp}_RM=\{\mathfrak{p}\in\mathrm{Spec}R\hspace{0.03cm}|\hspace{0.03cm}
\mathrm{Ext}^i_R(\kappa(\mathfrak{p}),M)\neq0\ \textrm{for\ some}\ i\}$,\end{center}where $\kappa(\mathfrak{p})=R_\mathfrak{p}/\mathfrak{p}R_\mathfrak{p}$.}}
\end{lem}
\begin{proof} First assume that $M$ is artinian. There is an injective resolution $M\rightarrow E^\bullet$ with $E^\bullet=0\rightarrow E(R/\mathfrak{m}_{01})\oplus\cdots\oplus E(R/\mathfrak{m}_{0n_0})\rightarrow E(R/\mathfrak{m}_{11})\oplus\cdots\oplus E(R/\mathfrak{m}_{1n_1})\rightarrow \cdots$ where each $\mathfrak{m}_{ij}\in\mathrm{Max}R$. Hence $\mathrm{Ext}^i_R(\kappa(\mathfrak{p}),M)=\mathrm{H}^i(\mathrm{Hom}_R(\kappa(\mathfrak{p}),E^\bullet))\cong
\mathrm{H}^i(\mathrm{Hom}_R(R_\mathfrak{p},\mathrm{Hom}_R(R/\mathfrak{p},E^\bullet)))\cong
\mathrm{Hom}_R(R_\mathfrak{p},\mathrm{H}^i(\mathrm{Hom}_R(R/\mathfrak{p},E^\bullet)))\cong
\mathrm{Hom}_R(R_\mathfrak{p},\mathrm{Ext}^i_R(R/\mathfrak{p},M))$ by \cite[3.5]{Ma} and \cite[Corollary 2.5]{CNh}. Consequently, we have the following equivalences: \begin{center}$\begin{aligned}\mathfrak{p}\in\mathrm{cosupp}_RM
&\Longleftrightarrow\mathfrak{p}\in\mathrm{Cosupp}_R\mathrm{Ext}^i_R(R/\mathfrak{p},M)\ \textrm{for\ some}\ i\\
&\Longleftrightarrow\mathfrak{p}\in\mathrm{Cos}_R\mathrm{Ext}^i_R(R/\mathfrak{p},M)\ \textrm{for\ some}\ i\\
&\Longleftrightarrow\mathrm{Ext}^i_R(\kappa(\mathfrak{p}),M)\neq0\ \textrm{for\ some}\ i,\end{aligned}$\end{center}
where the first one is by the remark after \cite[Theorem 3.8]{Y}, the second one is Lemma \ref{lem:0.1} and the last one is by the preceding isomorphism. Now assume that $M$ is a semi-discrete linearly compact $R$-module. Then \cite[Theorem]{Z} yields a short exact sequence \begin{center}$0\rightarrow N\rightarrow M\rightarrow A\rightarrow0$,\end{center}
where $A$ is artinian and $N$ is finitely generated. If $\mathfrak{p}\in\mathrm{cosupp}_RN$, then $\mathfrak{p}$ is a maximal ideal by \cite[Theorem 2.10]{Y}, and so $\mathrm{Ext}^i_R(\kappa(\mathfrak{p}),N)\neq0$ for some $i$. Therefore, \begin{center}$\begin{aligned}\mathfrak{p}\in\mathrm{cosupp}_RM
&\Longleftrightarrow\mathrm{Ext}^i_R(R/\mathfrak{p},{^\mathfrak{p}}N)\neq0\ \textrm{or}\ \mathrm{Ext}^j_R(R/\mathfrak{p},{^\mathfrak{p}}A)\neq0\ \textrm{for\ some}\ i,j\\
&\Longleftrightarrow\mathrm{Ext}^i_R(\kappa(\mathfrak{p}),N)\neq0\ \textrm{or}\ \mathrm{Ext}^j_R(\kappa(\mathfrak{p}),A)\neq0\ \textrm{for\ some}\ i,j\\
&\Longleftrightarrow\mathrm{Ext}^i_R(\kappa(\mathfrak{p}),M)\neq0\ \textrm{for\ some}\ i,\end{aligned}$\end{center}
We have show the desired equivalence.
\end{proof}

\begin{rem}\label{lem:0.3}{\rm The notion of `small' cosupport for a semi-discrete linearly compact $R$-module $M$ is the same as the one introduced by Benson, Iyengar and Krause \cite{BIK2}.}
\end{rem}

\bigskip
\section{\bf Cosupport of artinian modules}
 In this section we give the proof of the theorem in introduction. First we prove several lemmas
that are crucial for our argument.

\begin{lem}\label{lem:3.4'}{\it{Let $\mathfrak{p}$ be a non-maximal prime ideal of $R$ and $M$ an artinian $R$-module. One has an inequality
\begin{center}$\mathrm{Ass}_R \mathrm{Hom}_R(R_\mathfrak{p},M)\subseteq\{\mathfrak{q}\in\mathrm{Cosupp}_RM\hspace{0.03cm}|\hspace{0.03cm}\mathfrak{q}\subseteq\mathfrak{p}\}$.\end{center}}}
\end{lem}
\begin{proof} Since $M$ is an artinian $R$-module, it follows from \cite[Theorem 2.13]{Y} that $M$ has a composition series
\begin{center}$0=M_n\subset M_{n-1}\subset\cdots\subset M_1\subset M_0=M$,\end{center}where $M_{i-1}/M_i$ is cocyclic and $\mathfrak{p}_i=(0:_RM_{i-1}/M_i)\in\mathrm{Spec}R$ for $i=1,\cdots,n$. We use induction on $n$. If $n=1$ then $M=M_0/M_1$ is $\mathfrak{p}_1$-secondary and so $\mathrm{Hom}_R(R_\mathfrak{p},M)$ is either 0 or $\mathfrak{p}_1$-secondary by \cite[Lemma 3.1]{CNh}.
When $\mathrm{Hom}_R(R_\mathfrak{p},M)=0$ and the assertion holds. Assume that $\mathrm{Hom}_R(R_\mathfrak{p},M)$ is $\mathfrak{p}_1$-secondary and let $\mathfrak{q}\in\mathrm{Ass}_{R}\mathrm{Hom}_R(R_\mathfrak{p},M)$. Then
\begin{center}$\mathfrak{p}_1=\mathrm{Rad}(0:_{R}\mathrm{Hom}_R(R_\mathfrak{p},M))
\subseteq\mathfrak{q}$.\end{center}
Consequently, $\mathfrak{p}_1\subseteq\mathfrak{q}\subseteq\mathfrak{p}$ and hence $\mathfrak{q}\in\mathrm{Cosupp}_RM$. Now assume that $n>1$. One has a short exact sequence
\begin{center}$0\rightarrow M_1\rightarrow M \rightarrow M/M_1\rightarrow0$,\end{center}where $M/M_1$ is $\mathfrak{p}_1$-secondary and $M_1$ has a composition series
$0=M_n\subset M_{n-1}\subset\cdots\subset M_1$ such that $M_{i-1}/M_i$ is cocyclic and $\mathfrak{p}_i=(0:_RM_{i-1}/M_i)\in\mathrm{Spec}R$ for $i=2,\cdots,n$. Also we have the following exact sequence \begin{center}$0\rightarrow\mathrm{Hom}_R(R_\mathfrak{p},M_1)\rightarrow\mathrm{Hom}_R(R_\mathfrak{p},M)\rightarrow \mathrm{Hom}_R(R_\mathfrak{p},M/M_1)\rightarrow0$.\end{center} Let $\mathfrak{q}\in\mathrm{Ass}_{R}\mathrm{Hom}_R(R_\mathfrak{p},M)$. Then $\mathfrak{q}\in\mathrm{Ass}_{R}\mathrm{Hom}_R(R_\mathfrak{p},M_1)$ or $\mathfrak{q}\in\mathrm{Ass}_{R}\mathrm{Hom}_R(R_\mathfrak{p},M/M_1)$. If $\mathfrak{q}\in\mathrm{Ass}_{R}\mathrm{Hom}_R(R_\mathfrak{p},M/M_1)$ then $\mathfrak{q}\subseteq\mathfrak{p}$ and $\mathfrak{q}\in\mathrm{Cosupp}_RM/M_1\subseteq\mathrm{Cosupp}_RM$ by the preceding proof. If $\mathfrak{q}\in\mathrm{Ass}_{R}\mathrm{Hom}_R(R_\mathfrak{p},M_1)$ then $\mathfrak{q}\subseteq\mathfrak{p}$ and $\mathfrak{q}\in\mathrm{Cosupp}_RM_1\subseteq\mathrm{Cosupp}_RM$ by the induction, as required.
\end{proof}

\begin{lem}\label{lem:3.4}{\it{Let $\mathfrak{p}$ be a non-maximal prime ideal of a local ring $(R,\mathfrak{m})$ and $M$ an artinian $R$-module. There exists an equality
\begin{center}$\mathrm{Ass}_R\mathrm{Hom}_R(R_\mathfrak{p},M)=
\{\mathfrak{q}\in\mathrm{Cosupp}_RM\hspace{0.03cm}|\hspace{0.03cm}\mathfrak{q}\subseteq\mathfrak{p}\}$.\end{center}}}
\end{lem}
\begin{proof} The left hand side is included in the right hand side by Lemma \ref{lem:3.4'}.
  On the other hand, let $\mathfrak{q}\in\mathrm{Cosupp}_RM$ and $\mathfrak{q}\subseteq\mathfrak{p}$.
It follows from \cite[7.5,6.2,5.5]{Ma} that $\mathrm{Hom}_R(R_\mathfrak{p},M)$ is a semi-discrete linearly compact $R$-module, so \cite[Theorem]{Z} yields a short exact sequence
\begin{center}$0\rightarrow N\rightarrow \mathrm{Hom}_R(R_\mathfrak{p},M)\rightarrow A\rightarrow0$,\end{center}where $N$ is finitely generated and $A$ is artinian. If $N$ is artinian, then $\mathrm{Hom}_R(R_\mathfrak{p},M)$ is artinian and so $\mathrm{Hom}_R(R/\mathfrak{q},\mathrm{Hom}_R(R_\mathfrak{p},M))\neq0$ by \cite[Theorem 4.3]{Y}. Assume that $N$ is not artinian. Then there is a non-maximal prime ideal $\mathfrak{u}$ in $\mathrm{Ass}_RN$, and so $\mathfrak{u}\in\mathrm{Ass}_{R}\mathrm{Hom}_R(R_\mathfrak{p},M)$ and $\mathfrak{u}\subseteq\mathfrak{p}$ by Lemma \ref{lem:3.4'}. Set
$\mathfrak{p}'\in\mathrm{V}(\mathfrak{u}+\mathfrak{q})$ and $\mathfrak{p}'\subseteq\mathfrak{p}$. Then $\mathfrak{m}\neq\mathfrak{p}'\in\mathrm{supp}_RN$ and $\mathrm{Ext}^\ast_{R_{\mathfrak{p}'}}(\kappa(\mathfrak{p}'),N_{\mathfrak{p}'})\neq0$, $\mathrm{Ext}^\ast_{R_{\mathfrak{p}'}}(\kappa(\mathfrak{p}'),A_{\mathfrak{p}'})=0$ by \cite[Remark 9.2]{BIK}.
Hence the exact sequence \begin{center}$\mathrm{Ext}^\ast_{R_{\mathfrak{p}'}}(\kappa(\mathfrak{p}'),N_{\mathfrak{p}'})\rightarrow \mathrm{Ext}^\ast_{R_{\mathfrak{p}'}}(\kappa(\mathfrak{p}'),\mathrm{Hom}_R(R_\mathfrak{p},M)_{\mathfrak{p}'})\rightarrow \mathrm{Ext}^\ast_{R_{\mathfrak{p}'}}(\kappa(\mathfrak{p}'),A_{\mathfrak{p}'})=0$\end{center}implies that $\mathrm{Ext}^\ast_{R}(R/\mathfrak{p}',\mathrm{Hom}_R(R_\mathfrak{p},M))\neq0$. Also $\mathfrak{q}\subseteq\mathfrak{p}'$, so $\mathrm{Ext}^\ast_{R}(R/\mathfrak{q},\mathrm{Hom}_R(R_\mathfrak{p},M))\neq0$. Since $M$ is artinian, there is an injective resolution $M\rightarrow E^\bullet$ with $E^\bullet=0\rightarrow E(R/\mathfrak{m})^{n_0}\rightarrow E(R/\mathfrak{m})^{n_1}\rightarrow\cdots$. Then $\mathrm{Hom}_R(R_\mathfrak{p},M)\rightarrow \mathrm{Hom}_R(R_\mathfrak{p},E^\bullet)$ is also an injective resolution.
Therefore, we have the following isomorphisms
\begin{center}$\begin{aligned}\mathrm{Ext}^\ast_{R/\mathfrak{q}}((R/\mathfrak{q})_\mathfrak{p},\mathrm{Hom}_R(R/\mathfrak{q},M))
&\cong\mathrm{Ext}^\ast_R(R_\mathfrak{p}/\mathfrak{q}R_\mathfrak{p},M)\\
&=\mathrm{H}^\ast(\mathrm{Hom}_R(R_\mathfrak{p}/\mathfrak{q}R_\mathfrak{p},E^\bullet))\\
&\cong\mathrm{H}^\ast(\mathrm{Hom}_R(R/\mathfrak{q},\mathrm{Hom}_R(R_\mathfrak{p},E^\bullet)))\\
&=\mathrm{Ext}^\ast_R(R/\mathfrak{q},\mathrm{Hom}_R(R_\mathfrak{p},M))\neq0.\end{aligned}$\end{center}
As
$\mathrm{Hom}_R(R/\mathfrak{q},M)$ is a nonzero artinian $R/\mathfrak{q}$-module by \cite[Theorem 4.3]{Y}, it follows from \cite[Theorem 2.4]{CNh} that $\mathrm{Hom}_{R/\mathfrak{q}}((R/\mathfrak{q})_\mathfrak{p},\mathrm{Hom}_R(R/\mathfrak{q},M))\neq0$. We may assume that $R$ is a domain,
i.e., we have to show that $(0)\in\mathrm{Ass}_R\mathrm{Hom}_R(R_\mathfrak{p},M)$. Let $x\in\mathfrak{m}\backslash\mathfrak{p}$ and let $\mathfrak{q}$ be a
prime ideal minimal over $xR$, so that $\mathrm{ht}\mathfrak{q}=1$ by Krull's principal ideal theorem. If $s\in R\backslash\mathfrak{p}$
and $t\in R\backslash\mathfrak{q}$, then $(s,t)R$ is neither contained in $\mathfrak{p}$ nor $\mathfrak{q}$, so $(s,t)R\nsubseteq\mathfrak{p}\cup\mathfrak{q}$, i.e. $u=as+bt\in(R\backslash\mathfrak{p})\cap(R\backslash\mathfrak{q})$ for some $a,b\in R$. Then $\frac{1}{st}=\frac{u}{stu}=\frac{a}{tu}+\frac{b}{su}\in R_\mathfrak{p}+R_\mathfrak{q}$.
This shows that $R_\mathfrak{p}+R_\mathfrak{q}$ is a subring of $R/\mathfrak{m}=\kappa$. If it is a proper
subring of $\kappa$, then it possesses a nonzero prime ideal $\mathfrak{u}$ such that $0\neq\mathfrak{u}\cap R\subseteq\mathfrak{q}$. But $\mathrm{ht}\mathfrak{q}=1$, so $\mathfrak{u}\cap R=\mathfrak{q}$. Also $\mathfrak{u}\cap R\subseteq\mathfrak{p}R_\mathfrak{p}\cap R=\mathfrak{p}$, we have $x\in\mathfrak{p}$ which is a contradiction. Then $R_\mathfrak{p}+R_\mathfrak{q}=\kappa$, and therefore $R_\mathfrak{p}/(R_\mathfrak{p}\cap R_\mathfrak{q})\cong \kappa/R_\mathfrak{q}\neq0$. If $af=0$, $a\neq0$, where $f\in\mathrm{Hom}_R(\kappa/R_\mathfrak{q},M)$, then $0=af(\kappa/R_\mathfrak{q})=f(a\kappa/R_\mathfrak{q})=f(\kappa/R_\mathfrak{q})$,
and so $f=0$. Consequently,
$(0)\in\mathrm{Ass}_R\mathrm{Hom}_R(R_\mathfrak{p}/(R_\mathfrak{p}\cap R_\mathfrak{q}),M)\subseteq\mathrm{Ass}_R \mathrm{Hom}_R(R_\mathfrak{p},M)$.
\end{proof}

We are now ready to prove our main result:
\vspace{2mm}

\hspace{-0.4cm}{\it{Proof of Theorem.}} It is enough to show that $\mathrm{Cosupp}_RM\subseteq\mathrm{cosupp}_RM$. Let $\mathfrak{p}\in\mathrm{Cosupp}_RM$. If $\mathfrak{p}=\mathfrak{m}$, then $\mathfrak{m}\in\mathrm{Supp}_RM=\mathrm{supp}_RM$ and hence $\mathfrak{m}\in\mathrm{cosupp}_RM$ by \cite[Theorem 4.13]{BIK2} and Remark \ref{lem:0.3}. Now assume that $\mathfrak{p}\neq\mathfrak{m}$.
It follows from Lemma \ref{lem:3.4} that $\mathfrak{p}R_\mathfrak{p}\in\mathrm{Ass}_{R_\mathfrak{p}}\mathrm{Hom}_R(R_\mathfrak{p},M)$. Hence \cite[Theorem 4.3]{Y} implies that $\mathrm{Hom}_{R_\mathfrak{p}}(R_\mathfrak{p}/\mathfrak{p}R_\mathfrak{p},\mathrm{Hom}_R(R_\mathfrak{p},M))\cong
\mathrm{Hom}_{R}(R_\mathfrak{p}/\mathfrak{p}R_\mathfrak{p},M)\neq0$, that is to say, $\mathfrak{p}\in\mathrm{cosupp}_RM$.
\hfill{$\square$}

\begin{rem}\label{lem:2.23}{\rm If $R$ is not local then the equality in Theorem does not hold  in general. In fact, let $R$ be a domain with $\mathrm{Max}R=\{\mathfrak{m},\mathfrak{n}\}$. Set
 $M=E(R/\mathfrak{m})$. Then $\mathrm{Cosupp}_RM=\mathrm{Spec}R$. Note that $\mathrm{supp}_RM=\mathrm{Supp}_RM=\{\mathfrak{m}\}$, it follows from \cite[Theorem 4.13]{BIK2} that $\mathfrak{n}\not\in\mathrm{cosupp}_RM$. Consequently, $\mathrm{cosupp}_RM\subsetneq\mathrm{Cosupp}_RM$.}
\end{rem}

\begin{cor}\label{lem:2.22}{\it{Let $\mathfrak{a}$ be a proper ideal of a local ring $(R,\mathfrak{m})$.
 If $M$ is an artinian $R$-module, then $\mathrm{Rad}(0:_R\mathrm{Hom}_R(R/\mathfrak{a},M))=\mathrm{Rad}(\mathfrak{a}+(0:_RM))$.}}
\end{cor}
\begin{proof} Since $\mathrm{Hom}_R(R/\mathfrak{a},M)$ is artinian, it follows from Theorem that
\begin{center}$\begin{aligned}\mathrm{V}(\mathrm{Ann}_R\mathrm{Hom}_R(R/\mathfrak{a},M))&=\mathrm{Cosupp}_R\mathrm{Hom}_R(R/\mathfrak{a},M)\\
&=\mathrm{cosupp}_R\mathrm{Hom}_R(R/\mathfrak{a},M)\\
&=\mathrm{V}(\mathfrak{a})\cap\mathrm{cosupp}_RM\\
&=\mathrm{V}(\mathfrak{a})\cap\mathrm{V}(\mathrm{Ann}_RM)\\
&=\mathrm{V}(\mathfrak{a}+\mathrm{Ann}_RM).\end{aligned}$\end{center}
Therefore, we obtain the desired equality.
\end{proof}

\begin{cor}\label{lem:3.7}{\it{Let $N$ be an finitely generated $R$-module. One has
\begin{center}$\mathrm{Cosupp}_R\mathrm{Hom}_R(N,M)\subseteq\mathrm{Supp}_RN\cap\mathrm{Cosupp}_RM$,\end{center}
 equality hold if $M$ is artinian.}}
\end{cor}
\begin{proof} The inclusion is by \cite[Theorem 3.4]{Y}. If $M$ is artinian, then $\mathrm{Hom}_R(N,M)$ is also artinian. Hence Theorem implies that  \begin{center}$\mathrm{Cosupp}_R\mathrm{Hom}_R(N,M)=\mathrm{cosupp}_R\mathrm{Hom}_R(N,M)=\mathrm{supp}_RN\cap\mathrm{cosupp}_RM$,\end{center}
so the proof are complete.
\end{proof}

\bigskip \centerline {\bf ACKNOWLEDGEMENTS} This research was partially supported by National Natural Science Foundation of China (11761060,11901463).

\end{document}